\documentclass[12pt,a4paper]{article}

\usepackage{authblk}
\usepackage[margin=3cm]{geometry}
\usepackage{t1enc}
\usepackage[utf8]{inputenc}
\usepackage{amsthm,amsmath,amssymb}
\usepackage{graphicx}
\usepackage{enumerate}
\usepackage{hyperref}
\usepackage{bm}
\usepackage{comment}
\usepackage{amsfonts}
\usepackage{graphicx,caption}
\usepackage{bm}
\usepackage{amsmath, amsthm, amssymb}
\usepackage{graphicx}
\usepackage{hyperref}
\usepackage{relsize}
\usepackage{blkarray}
\usepackage{algpseudocode}

\usepackage{bbm}

\theoremstyle{plain}
\usepackage{amsthm}
\makeatletter
\newcommand{\newreptheorem}[2]{\newtheorem*{rep@#1}{\rep@title}\newenvironment{rep#1}[1]{\def\rep@title{#2 \ref*{##1}}\begin{rep@#1}}{\end{rep@#1}}}
\makeatother

\newtheorem{theorem}{Theorem}
\newtheorem*{theorem-non}{Theorem}
\newtheorem*{non-lemma}{Lemma}
\newtheorem{lemma}[theorem]{Lemma}
\newreptheorem{lemma}{Lemma}

\newtheorem{conjecture}[theorem]{Conjecture}
\theoremstyle{definition}

\DeclareMathOperator{\cok}{cok}

\DeclareMathOperator{\supp}{supp}

\DeclareMathOperator{\Aut}{Aut}

\DeclareMathOperator{\Disc}{Disc\,}

\DeclareMathOperator{\syst}{syst}

\begin{document}
\title{The $2$-torsion of determinantal hypertrees is not Cohen-Lenstra}
\author{Andr\'as M\'esz\'aros}
\date{}
\affil{University of Toronto}
\maketitle
\begin{abstract}
Let $T_n$ be a $2$-dimensional determinantal hypertree on $n$ vertices. Kahle and Newman conjectured that the $p$-torsion of $H_1(T_n,\mathbb{Z})$ asymptotically follows the Cohen-Lenstra distribution. For $p=2$, we disprove this conjecture by showing that given a positive integer $h$, for all large enough $n$, we have
\[\mathbb{P}(\dim H_1(T_n,\mathbb{F}_2)\ge h)\ge \frac{e^{-200h}}{(100h)^{5h}}.\]
We also show that $T_n$ is a bad cosystolic expander with positive probability. 
\end{abstract}

\section{Introduction}

\emph{Determinantal hypertrees} are natural higher dimensional generalizations of a uniform random spanning tree of a complete graph. They can be defined in any dimension, but in this paper, we restrict our attention to the $2$-dimensional case. A $2$-dimensional simplicial complex $S$ on the vertex set $[n]=\{1,2,\dots,n\}$ is called a ($2$-dimensional) hypertree, if
\begin{enumerate}[\hspace{30pt}(a)]
 \item\label{pra} $S$ has complete $1$-skeleton;
 \item\label{prb} The number of triangular faces of $S$ is ${n-1}\choose{2}$;
 \item\label{prc} The homology group $H_{1}(S,\mathbb{Z})$ is finite.
\end{enumerate}

Let $\mathcal{C}(n,2)$ be the set of hypertrees on the vertex set $[n]$. Kalai's generalization of Cayley's formula \cite{kalai1983enumeration} states that
\[\sum_{S\in \mathcal{C}(n,2)} |H_{1}(S,\mathbb{Z})|^2=n^{{n-2}\choose {2}}.\]

This formula suggests that the natural probability measure on the set of hypertrees is the one where the probability assigned to a hypertree $S$ is \begin{equation}\label{measuredef}
 \frac{|H_{1}(S,\mathbb{Z})|^2}{n^{{n-2}\choose {2}}}.
\end{equation}
It turns out that this measure is a determinantal probability measure \cite{lyons2003determinantal,hough2006determinantal}. Thus, a random hypertree $T_n$ distributed according to \eqref{measuredef} is called a determinantal hypertree. General random determinantal complexes were investigated by Lyons \cite{lyons2009random}. See \cite{kahle2022topology,meszaros2022local,werf2022determinantal,linial2019enumeration,meszaros2023coboundary,meszaros2024bounds} for some recent results on determinantal hypertrees.

Combining Kalai's formula with the trivial fact that \[|\mathcal{C}(n,2)|\le {{{n}\choose{3}}\choose{{n-1}\choose{2}}},\]
one can prove that $\mathbb{E}|H_1(T_n,\mathbb{Z})|=\exp(\Theta(n^2))$, see \cite{kahle2022topology}. Although the expected size of $H_1(T_n,\mathbb{Z})$ is large, the $p$-torsion of $H_1(T_n,\mathbb{Z})$ is believed to be of constant order. Kahle and Newman \cite{kahle2022topology} even had a candidate for the limiting distribution of the $p$-torsion.

\begin{conjecture}(Kahle and Newman \cite{kahle2022topology})\label{conj1}
Let $p$ be a prime. The $p$-torsion $\Gamma_{n,p}$ of $H_1(T_n,\mathbb{Z})$ converges to the \emph{Cohen-Lenstra distribution}. That is, for any finite abelian $p$-group $G$, we have
\[\lim_{n\to\infty} \mathbb{P}(\Gamma_{n,p}\cong G)=\frac{1}{|\Aut(G)|}\prod_{j=1}^{\infty}\left(1-p^{-j}\right).\]
\end{conjecture}

A similar conjecture was considered in \cite{kahle2020cohen} for the uniform measure on the set of hypertrees.

Conjecture~\ref{conj1} would imply that
\begin{equation}\label{rankconj}
 \lim_{n\to\infty} \mathbb{P}(\dim H_1(T_n,\mathbb{F}_p)=k)=p^{-k^2} \prod_{j=1}^{k} \left(1-p^{-j}\right)^{-2} \prod_{j=1}^{\infty}\left(1-p^{-j}\right),
\end{equation}
see \cite[Theorem 6.3]{cohen2006heuristics}.

See Section~\ref{sechistory} for more information on the Cohen-Lenstra heuristics. 

For a simplified model motivated by determinantal hypertrees and $p\ge 5$, the Cohen-Lenstra limiting distribution was established by the author \cite{meszaros2023cohen}.

\begin{theorem}\label{thm1}
Given a positive integer $h$, for all large enough $n$, we have
\[\mathbb{P}(\dim H_1(T_n,\mathbb{F}_2)\ge h)\ge \frac{e^{-200h}}{(100h)^{5h}}.\]
In particular, \eqref{rankconj} can not hold for all $k$ in the case $p=2$. Thus, Conjecture~\ref{conj1} is false.
\end{theorem}

Conjecture~\ref{conj1} remains open for $p>2$. Note that in the theory of cokernels of random matrices, there are a few examples where the case $p=2$ is different from the case $p>2$, see \cite{clancy2015cohen,meszaros2020distribution,meszaros2023cohen}. Thus, it is reasonable to think that Conjecture~\ref{conj1} might be true for $p>2$. We believe the following weaker conjecture should be true even for $p=2$.

\begin{conjecture}
Let $p$ be a prime. Then $\Gamma_{n,p}$ is tight, that is, given any $\varepsilon>0$, there is a $K$, such that
\[\mathbb{P}(|\Gamma_{n,p}|>K)<\varepsilon\text{ for all large enough }n.\]
\end{conjecture}

For some progress towards this conjecture, see \cite{meszaros2024bounds}.

Our proof of Theorem~\ref{thm1} is based on a second moment calculation. Note that a cochain in $C^1(T_n,\mathbb{F}_2)$ is uniquely determined by its support. Since $T_n$ has a complete $1$-skeleton, $C^1(T_n,\mathbb{F}_2)$ can be identified with the set of graphs on $[n]$.

Fix a positive integer $h$. Let $\mathcal{G}_n=\mathcal{G}_{n,h}$ be the set of graphs on the vertex set $[n]$ consisting of $h$ vertex disjoint $5$-cycles
and $n-5h$ isolated vertices.

Let us consider the random variable 
\[X_n=|\{G\in \mathcal{G}_n\,:\, G\in Z^1(T_n,\mathbb{F}_2)\}|. \]

\begin{theorem}\label{thmmoment}
For all large enough $n$, we have 
  \[\mathbb{E}X_n\ge \frac{e^{-100h}}{h!},\qquad\text{ and }\qquad\mathbb{E}X^2_n\le \frac{(100h)^{5h}}{(h!)^2}.\]
\end{theorem}

By the Paley–Zygmund inequality, for all large enough $n$, we have
\begin{equation}\label{Paley}
\mathbb{P}(X_n>0)\ge \frac{(\mathbb{E}X)^2}{\mathbb{E}X^2}\ge \frac{e^{-200h}}{(100h)^{5h}}.
\end{equation}

On the event $X_n>0$, we have $\dim H_1(T_n,\mathbb{F}_2)\ge h$, see Lemma~\ref{largedim}. Thus, we obtain Theorem~\ref{thm1}, see Section~\ref{secfinish} for details.

Kahle and Newman \cite{kahle2022topology} also asked about the expansion properties of determinantal hypertrees. We can define the expansion of a graph in several equivalent ways. Each of these definitions suggests a higher dimensional generalization. However, in higher dimensions these definitions are no longer equivalent, so in higher dimensions, we have several notions of expansion \cite{lubotzky2018high}. Vander Werf \cite{werf2022determinantal} proved that for any $\delta>0$, the fundamental group of the union of $\delta \log(n)$ independent determinantal hypertrees on $n$ vertices has property $(T)$ with high probability. The author~\cite{meszaros2023coboundary} proved that for any large enough  $k$, the union of $k$ independent determinantal hypertrees on $n$ vertices is a coboundary expander with high probability. Coboundary expansion was defined (implicitly) by Gromov~\cite{gromov2010singularities} and Linial, Meshulam~\cite{linial2006homological}. The somewhat weaker notion of \emph{cosystolic expansion} was introduced in~\cite{dotterrer2018expansion,kaufman2016isoperimetric,evra2016bounded}. It is a natural question to ask whether determinantal hypertrees are cosystolic expanders or not. It follows easily from our results that they are not.

Now we define cosystolic expanders following \cite{lubotzky2018high}. Let $K$ be a $d$-dimensional complex. We introduce a norm on $C^i(K,\mathbb{F}_2)$. We first define the weight $w(\sigma)$ of an $i$-dimensional face $\sigma$ as the number of top dimensional faces containing $\sigma$ divided by a normalizing factor, that is,
\[w(\sigma)=\frac{1}{{{d+1}\choose{i +1}}|K(d)|}|\{\tau\in K(d)\,:\,\sigma\subset \tau\}|.\]
The normalizing factor is chosen such that the total weight of the $i$-dimensional faces is $1$.

Then, for $f\in C^i(K,\mathbb{F}_2)$, we define
\[\|f\|=\sum_{\sigma\in\supp f} w(\sigma).\]

We can also extend this norm to cosets of $Z^i(K,\mathbb{F}_2)$ by setting
\[\|f+Z^i(K,\mathbb{F})\|=\min_{g\in f+Z^i(K,\mathbb{F})} \|g\|.\]

Next, for $0\le i\le d-1$, we define 
\begin{align*}\label{hidef}
  \tilde{\epsilon}_i(K)&=\min_{f\in C^i(K,\mathbb{F}_2)\setminus Z^i(K,\mathbb{F}_2)} \frac{\|\delta_i f\|}{\|f+Z^i(K,\mathbb{F}_2)\|}, \text{ and }\\
  \syst^i(K)&=\min_{f\in Z^i(K,\mathbb{F}_2)\setminus B^i(K,\mathbb{F}_2) } \|f\|.
\end{align*}

For $\varepsilon>0$, we say that $K$ is an $\varepsilon$-cosystolic expander if we have $\tilde{\epsilon}_i(K)\ge \varepsilon$ and $\syst^i(K)\ge \varepsilon$ for all $0\le i\le d-1$. The next theorem shows that with positive probability $T_n$ is a bad cosystolic expander.

\begin{theorem}\label{thmcosys}
We have
\[\liminf_{n\to\infty} \mathbb{P}\left(\syst^1(T_n)\le \frac{7}{n^2}\right)>0.\]
\end{theorem}

\subsection{Background and discussion}\label{sechistory}

\begin{conjecture}[Cohen and Lenstra \cite{cohen2006heuristics}]\label{coh}
  Let $p$ be a prime, and let $C_B$ be the Sylow p-subgroup of the class group of a uniform random imaginary quadratic field $K$ with $|{\Disc} K|\le B$. Then for any finite abelian $p$-group $G$, we have
  \[\lim_{B\to\infty} \mathbb{P}(C_B\cong G)=\frac{1}{|\Aut(G)|}\prod_{j=1}^{\infty}\left(1-p^{-j}\right).\]
\end{conjecture}

The distribution on the finite abelian $p$-groups appearing on the right hand side is called Cohen-Lenstra distribution.

Conjecture~\ref{coh} is still open. However, it turned out that it is easier to establish the Cohen-Lenstra limiting behaviour in the setting of random matrix theory. The first result of this kind is by Friedman and Washington \cite{friedman1987distribution} who proved that the cokernel of a Haar uniform square matrix over the $p$-adic integers asymptotically Cohen-Lenstra distributed.\footnote{For an $n\times n$ matrix $M$ over a principal ideal domain $R$, then the cokernel $\cok(M)$ of $M$ is defined as $R^n$ factored out by the $R$-submodule generated by the rows of $M$.} Because of the very algebraic and symmetric setting, this statement can be proved by purely algebraic tools. The lack of a finite Haar-measure on $\mathbb{Z}$ made it more challenging to obtain similar results over the integers. The breakthrough came from Wood, who combined algebraic and analytic tools to develop a version of the \emph{moment method} for abelian groups \cite{wood2017distribution}. Given a random abelian group $X$ and a deterministic finite abelian group $V$, the $V$-moment of $X$ is defined as the expected number of homomorphisms from $X$ to $V$.\footnote{Usually it is more convenient to only count surjective homomorphisms, but this will not be important for our discussion here.} Wood proved that the distribution of a random abelian $p$-group is uniquely determined by its moments provided that they do not grow too fast. See also \cite{sawin2022moment,van2024symmetric,wood2022probability} for further results on the moment problem.

Then for a large class of random matrices over $\mathbb{Z}$, Wood \cite{wood2019random} established Cohen-Lenstra limiting distribution for the Sylow $p$-subgroup of the cokernel by showing that the asymptotic moments of the cokernel match the moments of the Cohen-Lenstra distribution. The calculation of moments was carried out using discrete Fourier-transform. The theorem above applies for all square matrices with independent entries, where the entries are non-degenerate in a certain sense. Note that the non-degeneracy condition implies these matrices are dense, that is, a positive fraction of the entries is non-zero.

In recent years the moment method was successfully used in many situations \cite{recent1,recent2,recent3,recent4,recent5,recent6,recent7,recent8,recent9} see also the survey of Wood \cite{wood2022probability}. Note that some of these results can be also extended to symmetric matrices, where in the limit we get some modified version of the Cohen-Lenstra distribution \cite{clancy2015cohen,clancy2015note}.

However, all the results above are about dense matrices. Also the entries are independent (or in the symmetric case the entries above the diagonal are independent.) The only results about sparse matrices were obtained by the author \cite{meszaros2020distribution,meszaros2023cohen}. For these models the entries are not independent any more.

In the paper \cite{kahle2020cohen}, it was conjectured that the $p$-torsion of the homology group of certain random simplicial complexes are also Cohen-Lenstra distributed. The homology group often can be obtained as the cokernel of some submatrix of the matrix of the boundary map of the complex. Thus, these questions fit into the topic of cokernels of random matrices. However, the matrices coming from random complexes are very different from the matrices that are usually considered in this topic. Firstly, these matrices are sparse. Secondly, the entries are really far from being independent. But maybe the third difference is the most significant obstacle to understanding the cokernels of these matrices. Namely, the matrices coming from random complexes have less distributional symmetries. Most random matrices $M$ mentioned earlier have the property that given a permutation matrix $P$, \begin{equation}\label{symmetry}PM\text{ (or in certain cases $PMP^T$) has the same distribution as $M$.}\end{equation} This is not the case for the matrices coming from random complexes. For example, the first homology group of a determinantal hypertree $T_n$ can be described by a matrix $M=J^r_{n,2}[T_n(2)]$, where the rows are indexed with the edges of a complete graph, see Lemma~\ref{Hdet}. For this matrix $M$ the distributional symmetry given in \eqref{symmetry} does not hold in general, it only holds if the permutation matrix $P$ is coming from a graph automorphism of the complete graph.\footnote{Here we assume that the columns of $M$ are listed in a uniform random order. Of course, one can ensure that \eqref{symmetry} holds in general by considering the rows in a uniform random order. However, since the rows are indexed in a natural and meaningful way, it is rather artificial and unhelpful here.} In particular, the joint distribution of two rows looks very different depending on whether the two corresponding edges are incident or not. 

For random matrix $M$ satisfying \eqref{symmetry}, usually, any moment of the cokernel of $M$ can be written as a sum where the number of terms is polynomial in the size of the matrix $M$, which often makes it possible to obtain the asymptotic value of the moment. For hypertrees it does not look possible to reduce the calculation of the moments to only polynomially many cases as we explain next.

By the universal coefficient theorem, the $V$-moment of $H_1(T_n,\mathbb{Z})$ can be expressed as $\mathbb{E}|H^1(T_n,V)|=|V|^{-(n-1)}\mathbb{E}|Z^1(T_n,V)|$. Let us restrict our attention to the case $V=\mathbb{F}_2$. Note that a cochain in $C^1(T_n,\mathbb{F}_2)$ is uniquely determined by its support. Since $T_n$ has a complete $1$-skeleton, $C^1(T_n,\mathbb{F}_2)$ can be identified with the set of graphs on $[n]$. Thus, the $\mathbb{F}_2$-moment of $H_1(T_n,\mathbb{Z})$ can be expressed, as
\begin{equation}\label{momentsum}
  2^{-(n-1)}\sum \mathbb{P}(G\in Z^1(T_n,\mathbb{F}_2)),
\end{equation}
where the summation is over all graphs on $[n]$. 

Note that $\mathbb{P}(G_1\in Z^1(T_n,\mathbb{F}_2))=\mathbb{P}(G_2\in Z^1(T_n,\mathbb{F}_2))$, when $G_1$ and $G_2$ are isomorphic graphs. Thus, one can rewrite the sum above as a sum over isomorphism classes of graphs. But this still results in a sum with $\exp(\Theta(n^2))$ terms. This is in contrast with a random matrix $M$ satisfying the symmetry condition \eqref{symmetry}, where a similar regrouping would result in a sum with a polynomial number of terms. 

The author considered a random matrix model very similar to the one coming from hypertrees but satisfying \eqref{symmetry}. For this model the Cohen-Lenstra limiting distribution was established for $p\ge 5$ \cite{meszaros2023cohen}. Thus, the main obstacle to understanding the cokernels of matrices coming from random complexes is the lack of enough symmetries.

In \cite{meszaros2024bounds}, the author proved that the sum in \eqref{momentsum} is $\exp(o(n^2))$ by basically grouping together not only isomorphic graphs, but also graphs which are close in the so-called cut norm. The argument relies on the theory of dense graph limits \cite{lovasz2006limits,lovasz2012large}, and the large deviation principle of Chatterjee and Varadhan \cite{chatterjee2011large}.

The $\mathbb{F}_2$-moment of $H_1(C,\mathbb{Z})$ can be also expressed as
\begin{equation}\label{sumiso}\sum \mathbb{P}(G\in Z^1(T_n,\mathbb{F}_2)),
\end{equation}
where the sum is over all graphs $G$ on $[v]$ such that $n$ is an isolated vertex of $G$.

Knowing how the Cohen-Lenstra limiting distribution is usually proved, one would expect that if Conjecture~\ref{conj1} was true, then the contribution of graphs with exactly $5$ edges to the sum \eqref{sumiso} is negligible, but this is false in our case. If one considers a typical graph $G$ with $5$ edges, then it will have an isolated edge. For such a graph $G$, it is easy to show that $\mathbb{P}(G\in Z^1(T_n,\mathbb{F}_2))=0$. Thus, these graphs do not contribute to the sum~\eqref{sumiso}, and this is the behaviour that one would expect. However, there are a few very structured graphs with $5$ edges, namely, the $5$-cyles, which have a significant contribution to the sum~\eqref{sumiso}. Thus, graphs with the same number of edges can show drastically different behaviour. This happens because the random matrices that we consider are very structured, and they do not have the distributional symmetry described in~\eqref{symmetry}.

Assuming that the $2$-torsion has a limiting distribution, what should it be? As an example, let us consider the limiting distribution of the cokernels of random symmetric matrices \cite{clancy2015cohen,clancy2015note,wood2017distribution}. There, using the fact that we have symmetric matrices, we can endow the cokernel with some additional algebraic structure, namely, with a perfect symmetric pairing. To define the modified version of the Cohen-Lenstra distribution that describes the limiting distribution of the cokernels of random symmetric matrices, one should also take into account the automorphisms of this extra algebraic structure. Maybe the $2$-torsion of hypertrees can be also endowed with some additional algebraic structure, which could help to find the limiting distribution.

\bigskip

\textbf{Acknowledgement:} 
The author was supported by the NSERC discovery grant of B\'alint Vir\'ag and the KKP 139502 project.

\section{Preliminaries}

Given a simplicial complex $S$, we use the notation $S(d)$ for the set of $d$-dimensional faces of $S$.

Let $M$ be a matrix. For a subset $A$ of the rows of $M$ and a subset $B$ of the columns of $M$, the corresponding submatrix of $M$ will be denoted by $M[A,B]$. If $B$ is the set of all columns, we use the notation $M[A,*]$. If $A$ is the set of all rows, we use the notation $M[*,B]$ or just simply $M[B]$.

For $d\ge 1$, let $J_{n,d}$ be a matrix indexed by ${{[n]}\choose {d}}\times {{[n]}\choose{d+1}}$ defined as follows. Let $\sigma=\{x_0,x_1,\dots,x_d\}\subset [n]$ such that $x_0<x_1<\dots<x_d$. For a $\tau\in {{[n]}\choose {d}}$, we set
\[J_{n,d}(\tau,\sigma)=\begin{cases}
(-1)^i&\text{if }\tau=\sigma\setminus\{x_i\},\\
0&\text{otherwise.}
\end{cases}
\]
Note that $J_{n,d}$ is just the matrix of the $d$th boundary map of the simplex on $[n]$.

Let
\[J_{n,d}^r=J_{n,d}\left[{{[n-1]}\choose{d}},*\right].\]

The next lemma was proved in \cite[Lemma 2]{kalai1983enumeration}.
\begin{lemma}\label{Hdet}
Let $C$ be a $2$-dimensional simplicial complex on the vertex set $[n]$ with complete $1$-skeleton and ${n-1}\choose 2$ triangular faces. Then $C$ is a hypertree if and only if $\det J^r_{n,2}[C(2)]\neq 0$. Moreover, if $C$ is a hypertree, then 
\[|H_1(T_n,\mathbb{Z})|=|\det J^r_{n,2}[C(2)]|,\]
and $H_1(T_n,\mathbb{Z})\cong \cok(J^r_{n,2}[C(2)])$.

\end{lemma}

Given a graph $G$ on the vertex set $[n]$, let $F_k(G)$ be the set of all triangular faces whose boundary contains exactly $k$ edges from $E(G)$, that is,
\begin{equation}\label{FkGdef}
  F_k(G)=\left\{\sigma\in {{[n]}\choose{3}}\,:\, |(\partial \sigma) \cap E(G)|=k\right\},
\end{equation}
where $\partial \{a,b,c\}=\{\{b,c\},\{c,a\},\{a,b\}\}$.

The proof of the next lemma is straightforward.
\begin{lemma}\label{lemmaZ}
Let $C$ be a simplicial complex on the vertex set $[n]$ with complete $1$-skeleton, and let $G$ be a graph on $[n]$. Then \[G\in Z^1(C,\mathbb{F}_2)\text{ if and only if }C(2)\subset F_0(G)\cup F_2(G).\]
\end{lemma}

We will also rely on the following estimate: For any $F\subset {{[n]}\choose{3}}$, we have
\begin{equation}\label{hadamard}\mathbb{P}(F\subset T_n)\le \left(\frac{3}n\right)^{|F|},\end{equation}
see for example \cite[Section 2]{kahle2022topology}.
\section{First moment}\label{secfirst}

Let $G$ be a graph on the vertex set $[n]$ consisting of $h$ vertex disjoint $5$-cycles and $n-5h$ isolated vertices. More formally, we assume that $G$ can be constructed as follows. Let
\[v_{i,j}\quad (1\le i\le h,\, 0\le j\le 4)\]
be $5h$ pairwise distinct vertices from $[n]$, and let $V_0=V_0(G)$ be the set of these $5h$ vertices. Let $G$ be the graph with vertex set $[n]$ and edge set
\[E(G)=\{v_{i,j}v_{i,j+1}\,:\,1\le i\le h,\, 0\le j\le 4\}.\]
Here the second index of $v_{i,j}$ is taken mod $5$, that is, $v_{i,5}=v_{i,0}$, $v_{i,6}=v_{i,1}$ and so on. Later, we will also need the graphs $G_i$ ($1\le i\le h$) defined as follows. The vertex set $G_i$ is $[n]$, and the edge set of $G_i$ is
\[E(G_i)=\{v_{i,j}v_{i,j+1}\,:\, 0\le j\le 4\}.\]
Less formally, $G_i$ is the $i$th $5$-cycle of $G$.

Let $F_k=F_k(G)$ be defined as in \eqref{FkGdef}.

Note that
\[F_2=\cup_{i=1}^h F_{2,i},\]
where
\[F_{2,i}=F_2(G_i)=\{v_{i,j}v_{i,j+1}v_{i,j+2}\,:\, 0\le j\le 4\}.\]
Let $\overline{G}$ be the graph on the vertex set $[n-1]$ and edge set
\[E(\overline{G})={{[n-1]}\choose{2}}\setminus E(G).\]
\begin{lemma}\label{lemma8}
Assume that $n\notin V_0$. Let $C$ be a $2$-dimensional simplicial complex with ${n-1}\choose{2}$ triangular faces and complete $1$-skeleton. Then $C$ is a hypertree such that $G\in Z^1(C,\mathbb{F}_2)$ if and only if
\[C(2)=F_2\cup C_0,\]
where 
\[C_0\subset F_0,\quad |C_0|={{n-1}\choose{2}}-5h,\text{ and } \det J^r_{n,2}[E(\overline{G}),C_0]\neq 0.\]

In this case,
\[|\det J^r_{n,2}[C(2)]|=2^h\left|\det J^r_{n,2}[E(\overline{G}),C_0]\right|. \]

\end{lemma}
\begin{proof}
Let $C$ be a hypertree such that $G\in Z^1(C,\mathbb{F}_2)$. By Lemma~\ref{lemmaZ}, we have $C(2)\subset F_0\cup F_2$. Consider the submatrix $J^r_{n,2}[E(G),C(2)]$ of $J^r_{n,2}[C(2)]$. All of the columns of this submatrix are equal to zero except the ones indexed by the elements of $F_2$. Since $C$ is a hypertrees, the $5h$ rows of $J^r_{n,2}[E(G),C(2)]$ are linearly independent, thus, it must have at least $5h$ non-zero columns. Since $|F_2|=5h$, we must have $F_2\subset C(2)$. Thus, $C(2)=F_2\cup C_0$,
for some
$C_0\subset F_0$ such that $|C_0|={{n-1}\choose{2}}-5h$. 

If for a triangular face $\sigma$, we have $(\partial \sigma)\cap E(G_i)\neq \emptyset$, then $(\partial \sigma)\cap E(G_k)= \emptyset$ for all $k\neq i$. Therefore, after reordering the columns and rows, $J^r_{n,2}[C(2)]$ is a block upper diagonal matrix, where the diagonal blocks are
\[J^r_{n,2}[E(\overline{G}),C_0]\text{ and }J^r_{n,2}[E(G_i),F_{2,i}] \quad (1\le i\le h).\]

Thus,
\begin{equation}\label{deteq1}|\det J^r_{n,2}[C(2)]|=|\det J^r_{n,2}[E(\overline{G}),C_0]|\prod_{i=1}^h |\det J^r_{n,2}[E(G_i),F_{2,i}|.
\end{equation}

Next we investigate the matrix $J^r_{n,2}[E(G_i),F_{2,i}]$. For simplicity, first we assume that \[(v_{i,0},v_{i,1},v_{i,2},v_{i,3},v_{i,4})=(5,1,2,3,4).\] Then the matrix matrix $J^r_{n,2}[E(G_i),F_{2,i}]$ is of the form
\[
\begin{blockarray}{cccccc}
&\{1,2,5\} & \{1,2,3\} & \{2,3,4\} & \{3,4,5\}&\{1,4,5\} \\
\begin{block}{c(ccccc)}
  \{1,5\} & -1 & 0 & 0  & 0 &  -1\\
  \{1,2\} & +1 & +1  & 0 & 0 & 0\\
  \{2,3\} & 0 & +1  & +1  & 0 & 0\\
  \{3,4\} & 0  & 0 & +1 & +1 &0\\
  \{4,5\} & 0  & 0 & 0  & +1  &+1\\
\end{block}
\end{blockarray}\quad.
 \]
A straightforward calculation gives that $|\det J^r_{n,2}[E(G_i),F_{2,i}]|=2$. For a general choice of $(v_{i,0},\dots,v_{i,4})$, we have a very similar matrix but some of the rows and columns might be multiplied with $-1$. Thus, $\det |J^r_{n,2}[E(G_i),F_{2,i}]|=2$ holds in general. Combining this with \eqref{deteq1}, we obtain that
 \[|\det J^r_{n,2}[C(2)]|=2^h\left|\det J^r_{n,2}[E(\overline{G}),C_0]\right|.\]
Thus, if $C$ is a hypertree, then $\det J^r_{n,2}[C(2)]\neq 0$, so $\det J^r_{n,2}[E(\overline{G}),C_0]\neq 0$. This finishes the proof of one direction of the lemma. The other direction also follows easily. 
\end{proof}

\begin{lemma}\label{Lemma9}
Assume that $n\notin V_0$. Then
\[\mathbb{P}(G\in Z^1(T_n,\mathbb{F}_2))=\frac{2^{2h} \det(M_G) }{n^{{n-2}\choose{2}}},\]
where $M_G=J^r_{n,2}[E(\overline{G}),F_0]) (J^r_{n,2}[E(\overline{G}),F_0])^T$.
\end{lemma}
\begin{proof}
Combining Lemma~\ref{lemma8} with the Cauchy-Binet formula, we have
\[\sum_{\substack{C\in \mathcal{C}(n,2)\\G\in Z^1(C,\mathbb{F}_2)}} |\det J^r_{n,2}[C]|^2=\sum_{\substack{C_0\subset F_0\\|C_0|={{n-1}\choose {2}}-5h}} 2^{2h}|\det J^r_{n,2}[E(\overline{G}),C_0]|^2=2^{2h} \det(M_G).\]
Combining this with \eqref{measuredef} and Lemma~\ref{Hdet}, the statement follows.
\end{proof}

\begin{lemma}\label{lemma10}
Assuming that $n$ is large enough, we have
\[\mathbb{P}(G\in Z^1(T_n,\mathbb{F}_2))\ge 2^{2h} n^{-5h} e^{-80h}\qquad\text{for all }G\in\mathcal{G}_n.\]
\end{lemma}
\begin{proof}
Let $G\in \mathcal{G}_n$. By symmetry we may assume that $n\not\in V_0$. Let $F_0^r=F_0\cap {{[n-1]}\choose{3}}$. Note that
\[M_G-I=J^r_{n,2}[E(\overline{G}),F_0^r] (J^r_{n,2}[E(\overline{G}),F_0^r])^T.\]
Note that since $\overline{G}$ is connected $B^1(\overline{G},\mathbb{R})$ has dimension $|V(\overline{G})|-1=n-2$. Also, $B^1(\overline{G},\mathbb{R})$ is contained in the kernel of $J^r_{n,2}[E(\overline{G}),F_0^r] (J^r_{n,2}[E(\overline{G}),F_0^r])^T$. Thus, $1$ is an eigenvalue of $M_G$ with multiplicity at least $n-2$. 

We have
\[J^r_{n,2}(J^r_{n,2})^T-nI=-(J_{n-1,1})^T J_{n-1,1},\]
see \cite[(2) in the proof of Lemma 3]{kalai1983enumeration}.

Thus, 
\begin{equation}\label{meq1}
  J^r_{n,2}[E(\overline{G}),*](J^r_{n,2}[E(\overline{G}),*])^T-nI=-(J_{n-1,1}[*,E(\overline{G})])^TJ_{n-1,1}[*,E(\overline{G})].
\end{equation}

Since ${{[n]}\choose {3}}$ is the disjoint union of $F_0$ and $F_1\cup F_2$, we have
\begin{multline}\label{meq2}
  J^r_{n,2}[E(\overline{G}),*](J^r_{n,2}[E(\overline{G}),*])^T\\=J[E(\overline{G}),F_0](J^r_{n,2}[E(\overline{G}),F_0])^T+J^r_{n,2}[E(\overline{G}),F_1\cup F_2](J^r_{n,2}[E(\overline{G}),F_1\cup F_2])^T.
\end{multline}
Let us introduce the notations \begin{align*}N_G&=(J_{n-1,1}[*,E(\overline{G})])^TJ_{n-1,1}[*,E(\overline{G})],\\M_G'&=J^r_{n,2}[E(\overline{G}),F_1\cup F_2](J^r_{n,2}[E(\overline{G}),F_1\cup F_2])^T.\end{align*} 
Combining \eqref{meq1} and \eqref{meq2}, we obtain that
\[M_G=nI-N_G-M_G'.\]
Note that $J_{n-1,1}[*,E(\overline{G})]$ has $n-1$ rows, which are linearly dependent since their sum is $0$. Thus, the rank of $J_{n-1,1}[\cdot,E(\overline{G})]$ is at most $n-2$. Therefore, $N_G$ has rank at most $n-2$. Thus,
\[\dim \ker N_G\ge {{n-1}\choose{2}}-5h-(n-2)={{n-2}\choose{2}}-5h.\]

Let us consider the graph on the vertex set $F_1\cup F_2$, where two distinct triangular faces $\sigma_1,\sigma_2\in F_1\cup F_2$ are connected if $\partial \sigma_1\cap \partial \sigma_2\cap E(\overline{G})\neq \emptyset$. For each $\tau\in E(\overline{G})$, there are at most $4$ faces $\sigma\in F_1\cup F_2$ such that $\tau\in \partial \sigma$. Since for each face $\sigma\in F_1\cup F_2$, we have $|(\partial \sigma)\cap E(\overline{G})|\le 2$, the maximum degree of this graph is at most $2(4-1)=6$. Therefore, this graph has a proper coloring with $7$ colors. That is, there is partition of $F_1\cup F_2$ into sets $L_1,L_2,\dots,L_7$ such that if $\sigma_1,\sigma_2\in L_i$, and $\sigma_1\neq \sigma_2$, then $\partial \sigma_1\cap \partial \sigma_2\cap E(\overline{G})= \emptyset$. Let \[Q_i=J^r_{n,2}[E(\overline{G}),L_i](J^r_{n,2}[E(\overline{G}),L_i])^T.\] Then $M_G'=\sum_{i=1}^7 Q_i$. Moreover, $Q_i$ is block diagonal matrix, where each block is either
\[(0),\quad (1),\quad \begin{pmatrix} 1&1\\1&1\end{pmatrix},\quad \text{or}\quad \begin{pmatrix} 1&-1\\-1&1\end{pmatrix}.\]
Thus, the operator norm of $Q_i$ is at most $2$. Thus, $M_G'$ has operator norm at most $14$. The rank of $M_G'$ is clearly at most $|F_1\cup F_2|\le 5hn$. It follows that $\dim (\ker N_G\cap \ker M_G')\ge \dim \ker N_G-5nh$. All the vectors in $\ker N_G\cap \ker M_G'$ are eigenvectors of $M_G$ with eigenvalue $n$. Moreover, for any vector in $v\in\ker N_G$, we have $v^TM_G v\ge (n-14)\|v\|_2^2$. Thus, by the Courant-Fischer theorem, we obtain the following statement: Let us consider the eigenvalues of $M_G$ in decreasing order. Then the first ${{n-2}\choose 2}-5h-5nh$ eigenvalues are all at least $n$, the next $5nh$ eigenvalues are all at least $n-14$. Finally, we also have at least $n-2$ eigenvalues equal to $1$. Assuming that $n>15$, these are at least ${{n-1}\choose2}-5h$ eigenvalues. Thus, we found all the eigenvalues of $M_G$. Thus,
\[\det M_G\ge n^{{{n-2}\choose 2}-5h-5nh} (n-14)^{5nh}. \]
Combining this Lemma~\ref{Lemma9}, we obtain that
\begin{align*}\mathbb{P}(G\in Z^1(T_n,\mathbb{F}_2))&\ge \frac{2^{2h}  n^{{{n-2}\choose 2}-5h-5nh} (n-14)^{5nh}}{n^{{n-2}\choose {2}}}\\&=2^{2h} n^{-5h}\left(1-\frac{14}{n}\right)^{5nh}\\&\ge 2^{2h} n^{-5h} e^{-80h}
\end{align*}
for all large enough $n$.
\end{proof}
Any elements of $\mathcal{G}_n$ can be described by a sequence of $5h$ pairwise distinct vertices $v_{i,j}$ ($1\le i\le h$, $0\le j\le 4$) like above. However, given a graph in $\mathcal{G}_n$ this choice of vertices is not unique. A simple argument gives that we have $10^h h!$ choices. Thus,
\[|\mathcal{G}_n|=\frac{1}{10^h h!}\prod_{i=0}^{5h-1}(n-i).\]Therefore, for all large enough $n$,
\begin{equation}\label{Gnsize}\frac{1}{10^h h!}n^{5h}\ge |\mathcal{G}_n| \ge \frac{1}{20^h h!}n^{5h}.\end{equation}

Then combining Lemma~\ref{lemma10} and \eqref{Gnsize}, we obtain the estimate in Theorem~\ref{thmmoment} on the first moment of $X_n$, that is,
\[\mathbb{E}X_n\ge |\mathcal{G}_n|2^{2h} n^{-5h} e^{-80h}\ge \frac{e^{-100h}}{h!}.\]

\section{Second moment}

Given $G_0\in \mathcal{G}_n$ and $k\ge 0$, let us define
\[\mathcal{G}_n(G_0,k)=\{G_1\in \mathcal{G}_n\,:\, |F_2(G_0)\cap F_2(G_1)|=k\}.\]
Recall that for $G\in \mathcal{G}_n$, $V_0(G)$ denotes the set of the $5h$ non-isolated vertices of~$G$. 
\begin{lemma}\label{lemmametszet}
If $G_1\in \mathcal{G}_n(G_0,k)$, then $|V_0(G_0)\cap V_0(G_1)|\ge k$.

\end{lemma}
\begin{proof}
Let
\[N=|\{(\sigma,v)\,:\, \sigma\in F_2(G_0)\cap F_2(G_1), v\in V_0(G_0)\cap V_0(G_1)\cap \sigma \}|. \]

If $\sigma\in F_2(G_0)\cap F_2(G_1)$, then $|\sigma|=3$ and we have $v\in V_0(G_0)\cap V_0(G_1)\cap \sigma $ for any $v\in\sigma$. Thus, $N=3|F_2(G_0)\cap F_2(G_1)|=3k$. Given a $v\in V_0(G_0)\cap V_0(G_1)$, we have three triangular faces $\sigma\in F_2(G_0)$ such that $v\in \sigma$, so we have at most three $\sigma\in F_2(G_0)\cap F_2(G_1)$ such that $v\in \sigma$. Therefore, $3k=N\le 3 |V_0(G_0)\cap V_0(G_1)|$.
\end{proof}

\begin{lemma}\label{Gnksize}
We have
\[|\mathcal{G}_n(G_0,k)|\le \frac{n^{5h-k}(5h)^k}{10^h h!} {{5h}\choose{k}}.\]

\end{lemma}
\begin{proof}
A graph $G_1\in \mathcal{G}_n(G_0,k)$ can be described by $5h$ pairwise distinct vertices $v_{i,j}$ ($1\le i\le h$, $0\le j\le 4$). By Lemma~\ref{lemmametszet} at least $k$ of these vertices must be from $V_0(G_0)$. There are at most ${{5h}\choose{k}}n^{5h-k}(5h)^k$ such choices of $(v_{i,j})$, so the lemma follows by the same argument as \eqref{Gnsize}.
\end{proof}

Let $G_0,G_1\in \mathcal{G}_n$. It follows from Lemma~\ref{lemma8} that on the event that $G_0\in Z^1(T_n,\mathbb{F}_2)$ and $G_1\in Z^1(T_n,\mathbb{F}_2)$ both holds, we have $F_2(G_0)\cup F_2(G_1)\subset T_n$. Thus, assuming that $G_1\in \mathcal{G}_n(G_0,k)$, we have
\begin{align}\mathbb{P}(G_0\in Z^1(T_n,\mathbb{F}_2), G_1\in Z^1(T_n,\mathbb{F}_2))\label{G0G1est}&\le \mathbb{P}(F_2(G_0)\cup F_2(G_1)\subset T_n)\\&\le \left(\frac{3}n\right)^{|F_2(G_0)\cup F_2(G_1)|}\nonumber\\&= \left(\frac{3}n\right)^{10h-k}\nonumber,\end{align}
where the second inequality follows from \eqref{hadamard}.

Observe that
\begin{align*}
\mathbb{E}X_n^2&=\sum_{G_0\in \mathcal{G}_n}\sum_{G_1\in \mathcal{G}_n} \mathbb{P}(G_0\in Z^1(T_n,\mathbb{F}_2), G_1\in Z^1(T_n,\mathbb{F}_2))\\
&=\sum_{G_0\in \mathcal{G}_n}\sum_{k=0}^{5h}\sum_{G_1\in \mathcal{G}_n(G_0,k)} \mathbb{P}(G_0\in Z^1(T_n,\mathbb{F}_2), G_1\in Z^1(T_n,\mathbb{F}_2)).
\end{align*}
Combining this with the estimates given in \eqref{G0G1est} and Lemma~\ref{Gnksize}, we have
\begin{align*}
\mathbb{E}X_n^2&\le \sum_{G_0\in \mathcal{G}_n}  \sum_{k=0}^{5h}  \frac{n^{5h-k}(5h)^k}{10^h h!} {{5h}\choose{k}} \left(\frac{3}n\right)^{10h-k}\\&\le \sum_{G_0\in \mathcal{G}_n}  \frac{3^{10h}}{ h!}n^{-5h}\sum_{k=0}^{5h} {{5h}\choose{k}} (5h)^k\\&=|\mathcal{G}_n| \frac{3^{10h}}{ h!}n^{-5h} (5h+1)^{5h}\\&\le \frac{(100h)^{5h}}{(h!)^2},
\end{align*}
obtaining the estimate in Theorem~\ref{thmmoment} on the second moment of $X_n$.

\section{Finishing the proof of Theorem~\ref{thm1}}\label{secfinish}

\begin{lemma}\label{largedim}
Let $G\in\mathcal{G}_{n,h}$. If for a hypertree $S\in \mathcal{C}(n,2)$, we have $G\in Z^1(S,\mathbb{F}_2)$, then $\dim H_1(S,\mathbb{F}_2)\ge h$.
\end{lemma}
\begin{proof}
As in Section~\ref{secfirst}, let $G_1,G_2,\dots,G_h$ be the $h$ $5$-cycles of $G$ (considered as graphs on the vertex set $[n]$). Note that $F_0(G)\cup F_2(G)\subset F_0(G_i)\cup F_0(G_i)$. Combining this with Lemma~\ref{lemmaZ}, we see that if $G\in Z^1(S,\mathbb{F}_2)$, then $G_i\in Z^1(S,\mathbb{F}_2)$ for all $i$. None of the non-trivial linear combinations of $G_1,G_2,\dots,G_h$ over $\mathbb{F}_2$ is a coboundary, which shows that $\dim H^1(S,\mathbb{F}_2)\ge h$. Since $ H^1(S,\mathbb{F}_2)$ and $H_1(S,\mathbb{F}_2)$ are isomorphic the statement follows.
\end{proof}

The first statement of Theorem~\ref{thm1} follows from \eqref{Paley} and Lemma~\ref{largedim}. 

We prove the second statement by contradiction. If \eqref{rankconj} were for all $k$, then
\begin{align*}\lim_{n\to\infty} \mathbb{P}(\dim H_1(T_n)\ge h)&=\sum_{k=h}2^{-k^2} \prod_{j=1}^{k} \left(1-2^{-j}\right)^{-2} \prod_{j=1}^{\infty}\left(1-2^{-j}\right)\\&\le 2^{-h^2}\left(\prod_{j=1}^{\infty} \left(1-2^{-j}\right)^{-2}\right)\sum_{i=0}^{\infty} 2^{-i}\\&=C2^{-h^2}\end{align*}
for $C=2\prod_{j=1}^{\infty} \left(1-2^{-j}\right)^{-2}$. For large enough $h$, we have $C2^{-h^2}<\frac{e^{-200h}}{(100h)^{5h}}$, which contradicts the first statement of Theorem~\ref{thm1}.

\section{The proof of Theorem~\ref{thmcosys}}

Let $G\in \mathcal{G}_{n,1}$. Let $\tau\in E(G)$. By Lemma~\ref{lemma8}, on the event $G\in Z^1(T_n,\mathbb{F}_2)$, there are exactly two triangular faces of $T_n$ which contain $\tau$. Therefore, on the event $G\in Z^1(T_n,\mathbb{F}_2)$, we have \[\|G\|=5\frac{2}{3{{n-1}\choose 2}}=\frac{20}{3(n-1)(n-2)}.\]
Thus,
\[ \mathbb{P}\left(\syst^1(T_n)\le \frac{20}{3(n-1)(n-2)}\right)\ge \mathbb{P}(X_n>0).\]
Combining this with \eqref{Paley} and the fact that $\frac{20}{3(n-1)(n-2)}<\frac{7}{n^2}$ for all large enough $n$, Theorem~\ref{thmcosys} follows. 
\bibliography{references}
\bibliographystyle{plain}

\bigskip

\noindent Andr\'as M\'esz\'aros, \\
Department of Computer and Mathematical Sciences, \\University of Toronto Scarborough, Canada,\\ {\tt a.meszaros@utoronto.ca}

\end{document}